\documentclass{amsart}

\usepackage[english]{babel}

\usepackage{tikz}

\usepackage{amsmath, amssymb}
\usepackage{amstext, amscd, amsfonts}
\usepackage{mathtools}
\usepackage{cancel}
\usepackage{mathrsfs}

\usepackage{amsthm}
\DeclareMathOperator{\var}{var}

\usepackage{tikz-cd}

\newtheorem{theorem}{Theorem}[section]
\newtheorem{theorem*}{Theorem}

\newtheorem{corollary}[theorem]{Corollary}

\newtheorem{lemma}[theorem]{Lemma}

\theoremstyle{definition}

\newtheorem{definition}[theorem]{Definition}

\newtheorem{remark}[theorem]{Remark}

\numberwithin{equation}{section}
\numberwithin{figure}{section}

\title[Decomposition of factor codes and unique relative equilibriums]{Decomposition of infinite-to-one factor codes and uniqueness of relative equilibrium states}
\author{Jisang Yoo}
\address{Ajou University, Suwon, South Korea}
\email{jisangy@kaist.ac.kr}
\thanks{This research was supported by Basic Science Research Program through the National Research Foundation of Korea(NRF) funded by the Ministry of Education(2012R1A6A3A01040839) and the National Research Foundation of Korea (NRF) grant funded by the MEST 2015R1A3A2031159.}

 % Symbolic dynamics; thermo, entropy

\begin{document}
\maketitle
\begin{abstract}
  We show that an arbitrary factor map $\pi:X \to Y$ on an irreducible subshift of finite type is a composition of a finite-to-one factor code and a class degree one factor code.
  Using this structure theorem on infinite-to-one factor codes, we then prove that any equilibrium state $\nu$ on $Y$ for a potential function of sufficient regularity lifts to a unique measure of maximal relative entropy on $X$. This answers a question raised by Boyle and Petersen (for lifts of Markov measures) and generalizes the earlier known special case of finite-to-one factor codes.
  \end{abstract}

\section{Introduction}

The usual setting for relative thermodynamic formalism starts with a fixed factor map between topological dynamical systems.
Since the non-relative case indicates that symbolic dynamical systems have the easiest thermodynamic properties, it makes sense to work things out for the relative case in symbolic systems first. Indeed, there has been some progress in this direction in recent years, with or without thermodynamic application in mind.

In this paper, we restrict our attention to infinite-to-one factor codes $\pi:X \to Y$ from irreducible SFTs (shifts of finite type) to sofic shifts, and finite-to-one factor codes between irreducible sofic shifts.
The structure of factor codes of the latter type is well understood now (Chapter 8 in \cite{LM}).
For example, each finite-to-one factor code $\pi$ can be associated with a number called degree $d$ which is the common number of points in the fiber $\pi^{-1}(y)$ for almost all $y\in Y$, and almost all fibers $\pi^{-1}(y)$ have a certain permutation structure in it.
The class of finite-to-one factor codes is important because irreducible sofic shifts can be classified up to change by finite-to-one factor codes, where the complete invariant is the topological entropy (Finite Equivalence Theorem). This resembles part of the Ornstein theory which classifies Bernoulli systems up to isomorphism, where the complete invariant is the measure theoretical entropy.
Finite-to-one factor codes are important also because each surjective cellular automaton is a finite-to-one factor code, and because the class of finite-to-one factor codes is the simplest nontrivial examples of principal extensions.

An early relative result on finite-to-one factor codes is a result on uniqueness of preimage of Markov measures by Tuncel~\cite{Tuncel1981conditional}.
He showed that for any $\pi:X \to Y$ finite-to-one factor code between mixing SFTs, each Markov measure $\nu$ on $Y$ lifts uniquely to an invariant measure on $X$ and the unique lift is a Markov measure.
This is a relative thermodynamic result because Markov measures are just equilibrium states for locally constant functions (equivalently, (invariant) Gibbs measures for such functions or g-measures for such $g$).

For infinite-to-one factor codes, there are usually infinitely many invariant measures $\mu$ on $X$ that project to the same Markov measure $\nu$ on $Y$ and in some cases, even uncountably many Markov measures $\mu$~\cite{BT-ITO}.

In \cite[Problem 3.16]{Boyle-Petersen-2009-hidden-markov-symbolic}, Boyle and Petersen raised the following question. Given a (possibly infinite-to-one) factor map $\pi:X \to Y$ between irreducible SFTs and a Markov measure $\nu$ on $Y$, is there a unique measure $\mu$ of maximal relative entropy over $\nu$ (meaning any lift of $\nu$ with maximal entropy among lifts of $\nu$) and does $\mu$ have full support?
The question of full support was answered positively in \cite{yoomaximal} by the author and its generalization to relative equilibrium states was also answered positively in \cite{ant2014compen} by Antonioli. The full support result is a consequence of the following more general phenomenon.
For each ergodic measure $\nu$ of full support (not necessarily an equilibrium state of a sufficiently regular potential function, let alone a Markov measure), 
all MMREs (i.e., measures of maximal relative entropy) over $\nu$ have full support \cite{yoomaximal}.

With a possibly non-Markov $\nu$, in which case there may be more than one MMREs over it, Petersen, Quas and Shin \cite{PQS-MaxRelEnt} showed that the number of ergodic MMREs is nonetheless finite.
They also found some easy-to-check sufficient condition on $\pi$ that guarantees uniqueness of MMREs over any fully supported ergodic $\nu$ rather than just over any Markov $\nu$.

Allahbakhshi and Quas defined an invariant (for factor codes) called class degree in \cite{all2013classdegrelmaxent}, generalizing the notion of degree for finite-to-one factor codes, and showed that the number of ergodic MMRes is bounded by class degree. This gives a broader sufficient condition on $\pi$, namely, having class degree one, for uniqueness of MMREs over any fully supported ergodic $\nu$.

In order to answer the question of uniqueness over Markov measures under arbitrary factor codes,
we decompose a factor code $\pi: X \to Y$ into two factor codes $\pi_1: X \to \widetilde{Y}$ and $\pi_2: \widetilde{Y} \to Y$ where we can apply earlier results for class degree one code and finite-to-one code to $\pi_1$ and $\pi_2$ respectively. This is the first main result of this paper and is in some sense a structure theorem of infinite-to-one factor codes. It reduces the study of arbitrary factor codes into that of finite-to-one factor codes and that of class degree one codes.
The second main result is the application of this decomposition theorem: proof of uniqueness of MMREs over Markov measures or over other equilibrium states of sufficiently regular functions.

In the following table, we list known and new conditions for uniqueness of MMREs.

\begin{center}
\begin{tabular}{|l|r|r|}
\hline
 & condition on $\pi$ & condition on $\nu$ \\
\hline
classical (Tuncel) & finite-to-one & regular equilibrium state \\
recent result & class degree one & full support \\
new result & none & regular equilibrium state \\
\hline
\end{tabular}
\end{center}

In the next section, we introduce necessary definitions and facts. In Section~\ref{sec:decomposition}, we prove the decomposition result.
In Section~\ref{sec:uniquenes}, we prove uniqueness of relative equilibrium states (more general than MMREs) over sufficiently regular equilibrium states.
\section{Definitions}

$(X, T)$ is a \emph{topological dynamical system} (or TDS, for short) if $T$ is a homeomorphism of a compact metric space $X$.
A map $\pi: X \to Y$ between two TDS $(X, T)$ and $(Y, S)$ is a \emph{factor map} if $\pi$ is continuous, surjective and equivariant (i.e. $\pi$ commutes $T, S$).
A \emph{factor code} is a factor map between shift spaces.

Given a factor map $\pi: X \to Y$ and a probability measure $\mu$ on $X$, we denote by $\pi \mu$ or $\mu \circ \pi^{-1}$ the pushforward measure (image measure) from $\mu$ under $\pi$. So $\pi\mu$ is a probability measure on $Y$ defined by $\pi\mu (B) = \mu (\pi^{-1}(B))$ for each Borel subset $B \subset Y$.

Shift spaces in this paper always mean one-dimensional two-sided shift spaces, i.e., subsystems of the full shift $\mathcal A^{\mathbb Z}$ with some finite alphabet $\mathcal A$.
The two-sided assumption is because the degree theory and class degree theory rely on it.
Not much is lost by the two-sided assumption, since two-sided irreducible sofic shifts and their one-sided versions are essentially interchangeable in terms of invariant measures, entropy and other thermodynamical notions.
A shift of finite type (SFT, for short) means a shift space defined by local rules of uniformly bounded range, or equivalently, a shift space defined by a finite set of forbidden blocks.
A sofic shift means a shift space with rules of space complexity $O(1)$, or to be precise, an image of a SFT under a factor code.

A shift space $X$ is \emph{irreducible} if each pair of $X$-words $u, v$ can be connected by some third $X$-word $w$ so that $uwv$ is also an $X$-word.
A shift space $X$ is \emph{mixing} if it is topologically mixing.
The shift map on a shift space $X$ is denoted by $\sigma_X$ or just $\sigma$.

Definitions of infinite-to-one factor codes, finite-to-one factor codes and degrees of finite-to-one factor codes are as follows. For the general theory of these notions, we refer to \cite{LM}.

\begin{definition}\cite{LM}
  A factor code $\pi: X \to Y$ is \emph{finite-to-one} if the fiber $\pi^{-1}(y)$ over each point $y\in Y$ is a finite set. Otherwise, it is called \emph{infinite-to-one}.
\end{definition}

To define the degree of a finite-to-one factor code, we need transitive points.
\begin{definition}
  Given an irreducible sofic shift $Y$, a point $y \in Y$ is \emph{right transitive} if its forward orbit is dense (in $Y$). A point $y\in Y$ is \emph{doubly transitive} if both its forward orbit and its backward orbit are dense.
\end{definition}
It is well known that the doubly transitive points of $Y$ form a residual subset of $Y$.
\begin{definition}\cite{LM}
  Given a finite-to-one factor code $\pi$ from an irreducible sofic $X$ onto a sofic shift $Y$, the degree of $\pi$ is defined to be the unique number $d\in \mathbb N$ such that $|\pi^{-1}(y)|=d$ for all doubly transitive point $y\in Y$. If $X$ is an SFT, then this number is also the minimum of $|\pi^{-1}(y)|$ over all $y \in Y$.
\end{definition}

Given an irreducible sofic shift $Y$, there is always an irreducible SFT $X$ and a factor code $\pi:X \to Y$ which has degree one. The minimal right resolving presentation of $Y$ is such an example.

Now we move on to definitions for infinite-to-one factor codes.

\begin{definition}\cite{all2013classdegrelmaxent}
  Given a (possibly infinite-to-one) factor code $\pi$ from an irreducible SFT $X$ onto a sofic shift $Y$, we may define an equivalence relation on $X$ as follows. For $x, x' \in X$, we say $x \to x'$ if $\pi(x)=\pi(x')=y$ for some $y\in Y$ and for each $n \in \mathbb N$ there is $x''\in \pi^{-1}(y)$ such that $x''_{(-\infty,n]} = x_{(-\infty,n]}$ and $x''_{[i, \infty)} = x'_{[i, \infty)}$ for some $i>n$. We say $x \sim x'$ if $x \to x'$ and $x' \to x$. The equivalence classes from the equivalence relation $\sim$ on $X$ are called \emph{transition classes}. For each $y\in Y$, the transition classes in $\pi^{-1}(y)$ will be called transition classes over $y$.
\end{definition}

For each $y\in Y$, the number of transition classes over $y$ is finite. The minimum of this number over all $y\in Y$ is called \emph{class degree} of $\pi:X\to Y$. This is also the number of transition classes over each right transitive point of $Y$ (Corollary 4.23 in \cite{all2013classdegrelmaxent}).

\begin{definition}\cite{all2013classdegrelmaxent}
  Let $\pi:X\to Y$ be a 1-block factor code from a 1-step SFT $X$ onto a sofic $Y$. Let $w = w_0 w_1 \cdots w_p$ be a $Y$-block of length $1+p$. Let $n$ be an integer with $0<n<p$. Let $M$ be a subset of $\pi^{-1}(w_n)$. We say an $X$-word $u \in \pi^{-1}(w)$ is \emph{routable through} $a\in M$ at time $n$ if there is a block $u' \in \pi^{-1}(w)$ such that $u'_0 =u_0$ and $u'_n = a$ and $u'_p = u_p$. A triple $(w,n, M)$ is called a \emph{transition block} of $\pi$ if every $u\in\pi^{-1}(w)$ is routable through a symbol of $M$ at time $n$. The cardinality of $M$ is called the \emph{depth} of the transition block $(w,n,M)$. A \emph{minimal transition block} is a transition block of minimal depth. The minimal depth is the same as the class degree if $X$ is irreducible.
\end{definition}

We will rely on the following unique routing property for minimal transition blocks.
\begin{lemma}\label{lem:uniquerouteword}
  [Lemma 4.1 in \cite{all2014structureclass}]
  Let $\pi$ be a factor code from an irreducible SFT $X$ onto a sofic $Y$.
  Suppose $\pi$ is 1-block and $X$ is 1-step.
  Let $(w,n,M)$ be a minimal transition block.
  Then each preimage of $w$ is routable through a unique symbol of $M$ at $n$.
\end{lemma}

We will also rely on the following properties of transition classes over right transitive points.
\begin{lemma}\label{lem:mutsep}
  [Theorem 4.4 in \cite{all2014structureclass}]
  Let $\pi$ be a factor code from an irreducible SFT $X$ onto a sofic $Y$.
  Suppose $\pi$ is 1-block and $X$ is 1-step.
  Let $y\in Y$ be right transitive.
  Then any two points from two distinct transition classes over $y$ are mutually separated.
\end{lemma}

The unique routing property for transition classes over right transitive points is as follows.
\begin{lemma}\label{lem:uniqueroutepoint}
  [Lemma 5.1 in \cite{all2014structureclass}]
  Let $\pi$ be a factor code from an irreducible SFT $X$ onto a sofic $Y$.
  Suppose $\pi$ is 1-block and $X$ is 1-step.
  Let $y\in Y$ be right transitive.
  Suppose $y_{[i,i+|w|)}=w$ for some $i$ and some minimal transition block $(w,n,M)$.
  Let $C$ be a transition class over $y$.
  Then there is a unique symbol $b\in M$ such that for each $x\in C$, $x_{[i,i+|w|)}$ is routable through $b$ at $n$.
\end{lemma}

\section{Decomposition of infinite-to-one factor codes}
\label{sec:decomposition}

To decompose an infinite to one factor code $\pi$ into $\pi_1$ and $\pi_2$ and verify that the two resulting factor codes have the desired properties, we need to establish convenient characterizations of those properties.

\begin{definition}
  A point $x\in X$ in a shift space is \emph{left-asymptotic} to $x'\in X$ if $d(\sigma^{-i}x, \sigma^{-i} x') \to 0$ as $i \to \infty$.
  A point $x$ is \emph{right-asymptotic} to $x' \in X$ if $d(\sigma^i x, \sigma^i x') \to 0$ as $i \to \infty$ or equivalently, if there is some $m$ such that $x_{[m,\infty)} = x'_{[m,\infty)}$.
\end{definition}

In the following lemma, we establish characterizations of class degree one factor codes that do not rely on any recoding assumption on the factor code.
\begin{lemma}\label{lem:classdeg1}
  Let $\pi$ be a factor code from an irreducible SFT $X$ onto a sofic shift $Y$. The following are equivalent.
  \begin{enumerate}
  \item The class degree of $\pi$ is 1.
  \item For each doubly transitive point $y \in Y$ and for each ordered pair $x,x'$ in the fiber $\pi^{-1}(y)$ there is $x'' \in \pi^{-1}(y)$ that is left asymptotic to $x$ and right asymptotic to $x'$.
  \item For each right transitive point $y \in Y$ and for each ordered pair $x,x'$ in the fiber $\pi^{-1}(y)$ there is $x'' \in \pi^{-1}(y)$ that is left asymptotic to $x$ and right asymptotic to $x'$.
  \item There is a doubly transitive point $y \in Y$ such that each ordered pair $x,x'$ in the fiber $\pi^{-1}(y)$ there is $x'' \in \pi^{-1}(y)$ that is left asymptotic to $x$ and right asymptotic to $x'$.
  \item There is a right transitive point $y \in Y$ such that each ordered pair $x,x'$ in the fiber $\pi^{-1}(y)$ there is $x'' \in \pi^{-1}(y)$ that is left asymptotic to $x$ and right asymptotic to $x'$.
  \end{enumerate}
\end{lemma}
\begin{proof}
Each of the five conditions is invariant under conjugacy and so we may assume that $\pi$ is 1-block and $X$ is 1-step.
Among the last four conditions, the seemingly strongest condition is (3) and the seemingly weakest is (5), so it suffices to show that (5) implies (1) and that (1) implies (3).

((5) $\to$ (1)): Suppose $y$ is a right transitive point satisfying the condition (5). Since $y$ is right transitive, distinct transition classes over it are mutually separated,
but then the condition forces all transition classes over it to be the same, i.e., there is only one transition class over $y$. Since the class degree of $\pi$ is the minimum number of transition classes over points in $Y$, the condition (1) follows.

((1) $\to$ (3)): Suppose class degree 1 and let $y \in Y$ a right transitive point and $x,x'$ an ordered pair in the fiber $\pi^{-1}(y)$. Since $y$ is right transitive, there is only one transition class over it.
Therefore $x, x'$ are in the same transition class over $y$ and so the conclusion of condition (3) follows from the definition of transition classes.
\end{proof}

In the following lemma, we establish similar characterizations of degree $d$ factor codes.
\begin{lemma}\label{lem:degd}
  Let $\pi$ be a factor code from an irreducible \emph{sofic} shift $X$ onto another $Y$ and let $d \in \mathbb N$. The following are equivalent.
  \begin{enumerate}
  \item The factor code $\pi$ is finite-to-one and its degree is $d$.
  \item For each doubly transitive point $y \in Y$, the fiber $\pi^{-1}(y)$ contains exactly $d$ points.
  % \item For each right transitive point $y \in Y$, the fiber $\pi^{-1}(y)$ contains exactly $d$ points.
  \item There is a doubly transitive point $y \in Y$ such that the fiber $\pi^{-1}(y)$ contains exactly $d$ points.
  % \item There is a right transitive point $y \in Y$ such that the fiber $\pi^{-1}(y)$ contains exactly $d$ points.
  \end{enumerate}
\end{lemma}
\begin{proof}
It is known that under the assumption of our lemma, $\pi^{-1}(y)$ being finite for all $y \in Y$ implies the seemingly stronger property that there is a uniform upper bound on the size of $\pi^{-1}(y)$ over all $y \in Y$ (Theorem 8.1.19  in \cite{LM}) and this property in turn implies the property that the finite size of $\pi^{-1}(y)$ over any \emph{doubly transitive} $y \in Y$ is the same and this size is defined to be the degree of $\pi$ (Corollary 9.1.14 in \cite{LM}). Therefore the condition (1) implies (2). The condition (2) trivially implies (3).

It remains to show that (3) implies (1). Suppose $y$ is a doubly transitive point satisfying the condition (3). It is enough to show that $\pi$ is finite-to-one. Let $\pi_R: X_R \to X$ be the minimal right resolving presentation of $X$. Since $\pi_R$ is finite-to-one and $\pi^{-1}(y)$ is finite, the fiber $(\pi\circ\pi_R)^{-1}(y)$ is finite. So $\pi\circ\pi_R$ is a factor code from an irreducible SFT $X_R$ to $Y$ with a finite fiber over some doubly transitive point.
Since $\pi\circ\pi_R$ is a factor code on an irreducible SFT with a finite fiber over at least one doubly transitive point,
$\pi\circ\pi_R$ must be finite-to-one.\footnote{Exercise 9.1.2 in \cite{LM}
  
}
Therefore $\pi$ is also finite-to-one.
\end{proof}

Now we are ready to prove the first main theorem.
\begin{theorem}
  Let $\pi$ be a factor code from an irreducible SFT $X$ onto a sofic shift $Y$. Let $c_\pi$ be its class degree. Then there is an irreducible sofic shift $\widetilde{Y}$ and factor codes $\pi_1: X \to \widetilde{Y}$ and $\pi_2: \widetilde{Y} \to Y$ such that $\pi = \pi_2 \circ \pi_1$ and $\pi_1$ has class degree 1 and $\pi_2$ is finite-to-one and has degree $c_\pi$.
\end{theorem}
\[
  \begin{tikzcd}
  X \ar[dd, "\pi"] \ar[rd, "\pi_1"] \\
     & \widetilde{Y} \ar[dl, "\pi_2"] \\
  Y
  \end{tikzcd}
\]
\begin{proof}
  We may assume that $X$ is a 1-step SFT and $\pi$ is a 1-block factor code. Let $(w,n,M)$ be a minimal transition block for $\pi$. We may assume that 0 is a symbol that is \emph{not} an element of $M$. Let $\widetilde{M}$ be the disjoint union of $M$ and $\{0\}$.
  We will define a subshift $\widetilde{Y} \subset Y \times {\widetilde{M}}^{\mathbb Z}$. 
  Let $p_1: Y \times {\widetilde{M}}^{\mathbb Z} \to Y$ and $p_2: Y \times {\widetilde{M}}^{\mathbb Z} \to {\widetilde{M}}^{\mathbb Z}$ be the projection maps. 

  \textbf{Stage 1}. We construct $\pi_1, \pi_2, \widetilde{Y}$ first.
  
  \begin{proof}
We define a sliding block code $\pi_1: X \to Y \times {\widetilde{M}}^{\mathbb Z}$ (whose image will be denoted by $\widetilde{Y}$) by defining its two projections, namely, sliding block codes $p_1\circ\pi_1: X \to Y$ and $p_2\circ\pi_1: X \to {\widetilde{M}}^{\mathbb Z}$. First, define $p_1\circ\pi_1 = \pi$. Next, define $p_2\circ\pi_1$ in the following way. For each $x \in X$ and $i\in \mathbb Z$, let $(p_2\circ\pi_1(x))_i = 0$ if $\pi(x)_{[i-n, i+|w|-n-1]} \neq w$, otherwise let $(p_2\circ\pi_1(x))_i$ be the unique symbol in $M$ that the word $x_{[i-n, i+|w|-n-1]}$ (which is a preimage of $w$ via $\pi$) is routable through (at $n$).

It is easy to check that $\pi_1$ just defined is a sliding block code. Let $\widetilde{Y}$ be its image. This image is an irreducible sofic shift in $Y \times {\widetilde{M}}^{\mathbb Z}$ because $X$ is an irreducible SFT. Define $\pi_2: \widetilde{Y} \to Y$ to be the restriction of the projection $p_1$. It is easy to check that $\pi = \pi_2\circ\pi_1$. Since the composition $\pi = \pi_2\circ\pi_1$ is surjective the map $\pi_2$ is surjective as well and hence $\pi_2$ is a factor code onto $Y$. We have obtained a decomposition $\pi = \pi_2\circ\pi_1$ into factor codes. It remains to show that the factor codes $\pi_1, \pi_2$ have the desired properties.
\end{proof}

\textbf{Stage 2}. We claim $\pi_1$ has class degree one.
\begin{proof}
Note $\pi_1$ may not be a 1-block code. By Lemma~\ref{lem:classdeg1}, it is enough to show that for each right transitive point $\widetilde{y} \in \widetilde{Y}$ and for each ordered pair $x,x' \in \pi_1^{-1}(\widetilde{y})$ there is $x'' \in\pi_1^{-1}(\widetilde{y})$ that is left asymptotic to $x$ and right asymptotic to $x'$.

Let $\widetilde{y} = (y, s) \in \widetilde{Y}$ be right transitive and let $x,x' \in \pi_1^{-1}(\widetilde{y})$.
The point $y \in Y$ is right transitive because it is the image of right transitive $\widetilde{y}$ under the factor code $\pi_2$.
From the definition of $\pi_1$ we have $x, x' \in \pi^{-1}(y)$. Let $J \subset \mathbb Z$ be the set of all $i$ for which $y_{[i-n, i+|w|-n-1]} = w$, or equivalently, the set of all $i$ for which $s_i \neq 0$. The set $J$ marks the occurrences of the block $w$ along $y$. The set $J$ is non-empty (in fact, infinite to the right) because $y$ is right transitive. Fix one $i_* \in J$.
From the definition of $\pi_1$, the two blocks $x_{[i_*-n, i_*+|w|-n-1]}$ and $x'_{[i_*-n, i_*+|w|-n-1]}$ (which are preimages of $w$ via $\pi$) are routable through the common symbol $s_{i_*} \in M$. Using this routing, we can obtain a point $x'' \in \pi^{-1}(y)$ that is left asymptotic to $x$ and right asymptotic to $x'$ and $x''_{i_*} = s_{i_*}$. Since different transition classes over $y$ (via $\pi$) must be mutually separated, $x,x',x''$ are in the same transition class over $y$.
Therefore, since $y$ is right transitive and $x, x',x''$ are in the same transition class, for each $i \in J$, the block $x''_{[i-n, i+|w|-n-1]}$ is routable through $s_i$, by Lemma~\ref{lem:uniqueroutepoint}. (Without using that lemma, by definition of $x''$, it is obvious that the block $x''_{[i-n, i+|w|-n-1]}$ is routable through $s_i$ for $i=i^*$ and for those $i\in J$ with $|i- i^*|\ge |w|$. The lemma takes care of the remaining case $0< |i-i^*|<|w|$ where the two occurrences of $w$ may overlap.)
Therefore, $\pi_1(x'') = (y, s)$ and the proof of $\pi_1$ having class degree 1 is complete.
\end{proof}

\textbf{Stage 3}. 
It remains to show that $\pi_2$ is finite-to-one and has degree $c_\pi$.
\begin{proof}
Let $y \in Y$ be doubly transitive. By Lemma~\ref{lem:degd}, we only need to show that $\pi_2^{-1}(y)$ contains exactly $c_\pi$ points.

Since $y \in Y$ is doubly transitive, there are exactly $c_\pi$ transition classes in $X$ over $y$. Fix $x^{(1)}, \dots, x^{(c_\pi)} \in \pi^{-1}(y)$ to be representatives of the distinct transition classes. We have $\pi_1(x^{(k)}) = (y, s^{(k)})$ for some $s^{(k)} \in {\widetilde{M}}^{\mathbb Z}$ for each $x^{(k)}$.
We will show that $(y, s^{(1)}), \dots, (y, s^{(c_\pi)})$ are distinct $c_\pi$ points in $\pi_2^{-1}(y)$ and that there are no other points in $\pi_2^{-1}(y)$.

Let $J \subset \mathbb Z$ be the set of all $i$ for which $y_{[i-n, i+|w|-n-1]} = w$. $J$ is bi-infinite because $y$ is doubly transitive.
For each $i \in J$, $(s^{(k)}_i)_{1\le k \le c_\pi}$ are distinct $c_\pi$ symbols in $M$ because transition classes over $y$ are mutually separated. Therefore, $(y, s^{(k)})$ are distinct $c_\pi$ points in $\pi_2^{-1}(y)$.

It remains to show that there are no other points in $\pi_2^{-1}(y)$.
Suppose $(y,s^*)$ is in $\pi_2^{-1}(y)$.
Since $\pi_1$ is onto, there is some $x^* \in X$ such that $\pi_1(x^*) = (y, s^*)$.
The point $x^*$ must belong to one of the $c_\pi$ transition classes in $\pi^{-1}(y)$.
We may assume that $x^*$ is in the same transition class as $x^{(1)}$ (with respect to $\pi$).
Therefore, since $y$ is right transitive, for each $i \in J$, $s^*_i$ and $s^{(1)}_i$ must be the same symbol in $M$, by Lemma~\ref{lem:uniqueroutepoint}.
Therefore $s^* = s^{(1)}$ and we have $(y,s^*) = (y, s^{(1)})$.
Since $(y,s^*)$ was arbitrarily chosen, we have shown that there are no points in $\pi_2^{-1}(y)$ other than $(y, s^{(1)}), \dots, (y, s^{(c_\pi)})$.
\end{proof}
We have shown that the two factor codes have the desired properties and this completes the proof of the theorem.
\end{proof}

\begin{definition}
  Let $\pi: X \to Y$ be a factor code from an irreducible SFT onto a sofic shift. 
Any irreducible sofic shift $\widetilde{Y}$ and factor codes $\pi_1, \pi_2$ satisfying the conclusion of the theorem above are called a \emph{class degree decomposition} of $\pi$. In this case, the sofic shift space $\widetilde{Y}$ is called a \emph{class degree factor} of $X$ over $Y$ with respect to $\pi$.
\end{definition}

Class degree decompositions are not unique up to conjugacy in general. Depending on the choice of the minimal transition block $w$, we may get different decompositions.

Since the occurrences of the block $w$ along $y$ may have unbounded gaps, the class degree factor constructed from $w$ is usually strictly sofic.
There are factor codes where all class degree factors are strictly sofic. Any factor code $\pi: X \to Y$ which cannot be decomposed into two factor codes with an SFT in the middle is such an example.

\section{Uniqueness of relative equilibrium states over regular equilibrium states}
\label{sec:uniquenes}

Let $(X, T)$ be a TDS. Then $M(X, T)$ denotes the set of all invariant (probability) measures on $X$. This set is a compact metrizable space under the weak star topology (same as the vague topology in our case).
If $T$ is understood (usually when $X$ is a shift space so that $T$ is the shift map $\sigma_X$ on $X$), then we denote it by $M(X)$.

Given a measure $\mu \in M(X, T)$ and a function $\phi \in C(X)$ (where $C(X)$ is the set of all continuous (real-valued) functions on $X$), we denote by $h(\mu, T)$ or $h(\mu)$ the measure-theoretical entropy of $\mu$ with respect to $T$. The expression $\mu(\phi)$ denotes the integral $\int \phi \,d\mu$.

Given a continuous function $\phi$ on a shift space $X$, we say $\phi$ is H\"{o}lder-continuous if $\var_n \phi \le C \alpha^n$ for some constants $C>0$ and $0<\alpha<1$, where $$\var_n \phi := \max\{ |\phi(x) - \phi(x')| \  : \  x_{[-n,n]} = x'_{[-n,n]}, \  x, x' \in X \}$$

For each $m \in \mathbb N$, we denote by $ S_m \phi$ the cocycle sum $\phi + \phi\circ T + \dots + \phi\circ T^{m-1}$.

We denote by $P(X, T, \phi)$ (or $P(T, \phi)$ if $X$ is understood, or even $P(\phi)$) the topological pressure of $\phi$ with respect to $T$.

We denote by $P(\mu, T, \phi)$ or $P(\mu, \phi)$ the measure pressure $h(\mu, T)+ \mu(\phi)$, i.e., the free energy of $\mu$ with respect to $\phi$.

\begin{definition}
  Let $\phi \in C(X)$ where $(X, T)$ is some TDS.
  Within the measures $\mu$ in $M(X, T)$, those measures maximizing the measure pressure $P(\mu, \phi)$ are called \emph{equilibrium states} for the potential function $\phi$.
  Equilibrium states for the constant function $\phi = 0$ are called \emph{measures of maximal entropy} or MMEs for short.
\end{definition}

The variational principle for pressure states that the supremum of the measure pressure $P(\mu, \phi)$ over all $\mu \in M(X, T)$ is the same as the topological pressure $P(\phi)$. Therefore, equilibrium states for $\phi$ are precisely those $\mu$ satisfying the equality $P(\mu, \phi) = P(\phi)$.

It is known that if $(X, T)$ is expansive then the pressure map $\mu \mapsto P(\mu, \phi)$ is upper semi-continuous on the compact space $M(X, T)$ and therefore achieves maximum. Therefore, in this case, there is at least one equilibrium state for $\phi$.
For mixing SFTs, there is a unique equilibrium state for $\phi$ if $\phi$ is sufficiently regular. This uniqueness property seems to have no name. Here we will call such systems \emph{strongly} intrinsically ergodic systems, because systems with unique measure of maximal entropy are called intrinsically ergodic systems.
\begin{definition}
  Let $X$ be a shift space. We say $X$ is \emph{strongly intrinsically ergodic} for H\"{o}lder continuous functions if there is a unique equilibrium state for each H\"{o}lder-continuous function $\phi$ on $X$.
\end{definition}
\begin{remark}
  In a journal-published version of this paper and future papers, \(X\) will be just called \emph{intrinsically ergodic for} the class of H\"{o}lder continuous functions, by dropping ``strongly''.
\end{remark}

First we establish that irreducible sofic shifts are strongly intrinsically ergodic. This probably is a folklore result, but for completeness we will include a proof which closely parallels a standard proof of the similar fact for Axiom A diffeomorphisms \cite{Bowen-EqAnosov}. The proof works by transferring strongly intrinsic ergodicity of mixing SFTs to sofic shifts. For that, we need a quick lemma about pressure:
\begin{lemma}
  Let $\pi: X^* \to X$ be a factor code between shift spaces. Let $\phi: X \to \mathbb R$ be continuous and let $\phi^* = \phi\circ\pi$. Then $P(\phi^*) \ge P(\phi)$. Equality holds if $\pi$ is finite-to-one.
\end{lemma}
\begin{proof}
  For each $\mu^* \in M(X^*)$ and $\mu= \pi \mu^*$, we have \[P(\mu^*, \phi^*) = P(\mu, \phi) + h(\mu^*|\mu)\]
  where $h(\mu^*|\mu) = h(\mu^*) - h(\mu) \ge 0$ is the relative entropy of $\mu^*$ with respect to $\pi$.
  Since the pushforward map $M(X^*) \to M(X)$ is surjective, we obtain the desired inequality by applying the variational principle on $\phi^*$ and $\phi$ each.
  Equality in finite-to-one case follows because $h(\mu^*|\mu)=0$ in that case.
\end{proof}
\begin{lemma}
  Let $X$ be an irreducible sofic shift. Let $\phi:X\to \mathbb R$ be H\"{o}lder continuous. Then $\phi$ has a unique equilibrium state.
\end{lemma}
\begin{proof}
  First we assume $X$ is an irreducible SFT. Then we have the spectral decomposition $X = \cup_{i=1}^m X_i$. where $X_i$ are disjoint from each other and $\sigma(X_i) = X_{i+1 \mod m}$ and $(X_1, \sigma^m)$ is conjugate to a mixing SFT.
  A measure $\mu \in M(X,\sigma)$ induces $\mu' \in M(X_1,\sigma^m)$ by restriction to $X_1$ and conversely, any $\mu' \in M(X_1,\sigma^m)$ induces a measure $\mu \in M(X,\sigma)$ as the convex combination of copies of $\mu'$ on $X_i$.
  Therefore $\mu \leftrightarrow \mu' $ is a bijection between $M(X,\sigma)$ and $M(X_1,\sigma^m)$. We have $h(\mu', \sigma^m) = m h(\mu, \sigma)$ and $\mu'(S_m\phi) = m \mu(\phi)$. Therefore, finding $\mu$ maximizing $P(\mu, \phi)$ is equivalent to finding $\mu'$ maximizing $P(\mu', S_m \phi)$. Since $S_m \phi$ restricted to $X_1$ is H\"{o}lder continuous, we are done.
  We have shown that any irreducible SFT is strongly intrinsically ergodic for H\"{o}lder continuous functions.

  The strictly sofic case remains.
  Let $\pi:X^* \to X$ be the minimal right resolving presentation of $X$. In particular, $X^*$ is an irreducible SFT and $\pi$ has degree one.
  Let $\phi^* = \phi\circ \pi$. The function $\phi^*: X^*\to \mathbb R$ is H\"{o}lder continuous.

  As the first part of this proof showed, we have a unique equilibrium state $\mu_{\phi^*}$ for $\phi^*$.
  Since $\mu_{\phi^*}$ is a fully supported ergodic measure on $X^*$, the set of doubly transitive points is a full measure set with respect to $\mu_{\phi^*}$.
  Let $\mu_\phi = \pi \mu_{\phi^*}$.
  Then $\mu_\phi$ is an invariant measure on $X$. The measure preserving systems arising from $\mu_\phi$ and $\mu_{\phi^*}$ are conjugate because $\pi$ is one-to-one except on a $\mu_{\phi^*}$-null set, namely, the complement of the set of the doubly transitive points in $X^*$. ($\pi$ is one-to-one on doubly transitive points because $\pi$ has degree one.)
  In particular, $h(\mu_\phi) = h(\mu_{\phi^*})$ and $\mu_\phi(\phi) = \mu_{\phi^*}(\phi^*)$. Therefore, we have
  \[ P(\mu_\phi, \phi) = P(\mu_{\phi^*}, \phi^*) = P(\phi^*) \ge P(\phi) \]
  Hence $\mu_\phi$ is an equilibrium state for $\phi$.

  Suppose $\mu$ is any other equilibrium state of $\phi$. Pick an invariant measure $\mu^*$ on $X^*$ with $\pi \mu^* = \mu$. Then $h(\mu^*) \ge h(\mu)$ and we have
  \[ P(\mu^*, \phi^*) \ge P(\mu, \phi) = P(\phi) = P(\phi^*) \]
  Therefore $\mu^*$ is an equilibrium state for $\phi^*$ and by uniqueness we have $\mu^* = \mu_{\phi^*}$. Then $\mu = \pi \mu_{\phi^*} = \mu_\phi$.
\end{proof}

Next, we establish the unique lift property of regular equilibriums via finite-to-one factor codes. Tuncel \cite{Tuncel1981conditional} prove this for Markov measures on SFTs, but the same proof works for H\"{o}lder continuous functions and irreducible sofic shifts. We reproduce the proof shortly in our notation in the following lemma.
\begin{lemma}\label{lem:uniquelift}
  Let $\pi: X\to Y$ be finite-to-one factor codes between two shift spaces that are strongly intrinsically ergodic (for H\"{o}lder continuous functions). Let $\psi$ be a H\"{o}lder continuous function on $Y$ and let $\mu_\psi$ be its unique equilibrium state. Then there is a unique invariant measure in $M(X)$ that projects to $\mu_\psi$. The unique measure is the unique equilibrium state for $\psi\circ \pi$.
\end{lemma}
\begin{proof}
  Since $\pi$ is finite-to-one, we have $P(\psi) = P(\psi\circ\pi)$.
  Let $\mu$ be the unique equilibrium state for $\psi\circ\pi$. This is unique because $\psi\circ\pi$ is H\"{o}lder continuous. Its image $\pi\mu$ is an equilibrium state for $\psi$ and so the image must be $\mu_\psi$.

  Let $\mu'$ be another measure in $M(X)$ whose image is $\mu_\psi$. Then \[P(\mu', \psi\circ\pi) = P(\mu_\psi, \psi) = P(\psi) = P(\psi\circ\pi).\]
  Therefore $\mu'$ is an equilibrium state for $\psi\circ \pi$ but the equilibrium state is unique.
\end{proof}

\begin{definition}
  Let $\pi: X \to Y$ be a factor map between two TDSs $(X, T)$ and $(Y,S)$. Let $\nu \in M(Y,S)$.
  A measure $\mu \in M(X,T)$ is called a \emph{measure of maximal relative entropy (MMRE)} over $\nu$ if it maximizes the entropy $h(\mu)$ subject to the constraint $\pi \mu = \nu$.
\end{definition}
\begin{definition}
  Let $\pi, \nu$ be as in the previous definition.
Let $\phi \in C(X)$.
A measure $\mu \in M(X,T)$ is called a \emph{relative equilibrium state} of $\phi$ over $\nu$ if it maximizes the measure pressure $P(\mu, \phi)$ subject to the same constraint $\pi\mu = \nu$.
\end{definition}

\begin{lemma}
  \label{lem:uniquemmre1}
  [special case of the main theorem in \cite{yoo2014releqclass}]
  Let $\pi_1: X \to \widetilde{Y}$ be a class degree one factor code from an irreducible SFT onto a sofic shift.
  Let $\widetilde{\nu}$ be a fully supported ergodic measure on $\widetilde{Y}$.
  Let $\phi$ be H\"{o}lder continuous on $X$.
  Then there is a unique relative equilibrium state of $\phi$ over $\widetilde{\nu}$.
\end{lemma}
We remark that the above result for the special case $\phi=0$ is an older result. This old special case alone already generates a new result if combined with our first main theorem.

We are ready to apply the lemmas so far to prove the second main theorem, which is a consequence of the first main theorem.
\begin{theorem}
  Let $\pi: X \to Y$ be a factor code from an irreducible SFT onto a sofic shift.
  Let $\phi, \psi$ be H\"{o}lder continuous functions on $X, Y$ respectively.
  Let $\nu \in M(Y)$ be the unique equilibrium state for $\psi$.
  Then there is a unique relative equilibrium state $\mu$ of $\phi$ over $\nu$.
\end{theorem}
\begin{proof}
  By expansivity of the shift map for $X$, the pressure map $\mu \mapsto P(\mu, \phi)$ is upper semi-continuous. Since the measure fiber $\pi^{-1}(\nu)$ is a compact subset of $M(X)$, there is at least one $\mu$ in it maximizing the measure pressure. In other words, there is at least one relative equilibrium state of $\phi$ over $\nu$.

  Let $\mu$ be any measure in $M(X)$ that projects to $\nu$.
    Let $\widetilde{Y}, \pi_1, \pi_2$ be a fixed class degree decomposition for $\pi$.
  Let $\widetilde{\nu}$ be the image of $\mu$ on $\widetilde{Y}$.
  Since $\mu$ projects to $\nu$ on $Y$, the measure $\widetilde{\nu}$ must project to the same measure $\nu$. By uniqueness in Lemma~\ref{lem:uniquelift}, $\widetilde{\nu}$ is the unique lift of $\nu$ to $\widetilde{Y}$. In particular, $\widetilde{\nu}$ does not depend on $\mu$.

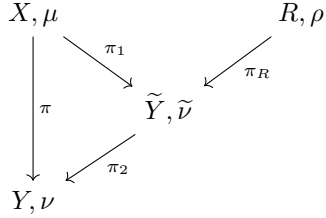
\[\begin{tikzcd}
  X, \mu \ar[dd, "\pi"] \ar[rd, "\pi_1"] &  & R, \rho \ar[dl, "\pi_R"] \\
     & \widetilde{Y}, \widetilde{\nu} \ar[dl, "\pi_2"] \\
  Y, \nu
\end{tikzcd}\]

  We claim that $\widetilde{\nu}$ is fully supported and ergodic. One way of showing this is to lift $\widetilde{\nu}$ to an equilibrium state $\rho$ on $R$ where $\pi_R: R\to \widetilde{Y}$ is the minimal right resolving presentation of $\widetilde{Y}$. This is done by applying Lemma~\ref{lem:uniquelift} to $\pi_2\circ\pi_R: R \to Y$ and $\nu$ so that $\rho$ is the equilibrium state for $\psi\circ\pi_2\circ\pi_R$.
  Being the unique equilibrium state for a H\"{o}lder continuous function on an irreducible SFT, the measure $\rho$ must be fully supported and ergodic. Therefore its image $\widetilde{\nu}$ is also fully supported and ergodic.

  The measure fiber on $X$ over $\nu$ and the measure fiber over $\widetilde{\nu}$ are the same subsets of $M(X)$. Therefore, $\mu$ is a relative equilibrium state of $\phi$ over $\nu$ if and only if it is a relative equilibrium state of $\phi$ over $\widetilde{\nu}$.
  But the latter relative equilibrium state is unique by Lemma~\ref{lem:uniquemmre1}, since $\pi_1$ has class degree one and $\widetilde{\nu}$ is fully supported.
\end{proof}

We state the special case $\phi=0$ of the above main theorem as follows.
\begin{corollary}\label{cor:answer}
  Let $\pi: X \to Y$ be a factor code from an irreducible SFT onto a sofic shift.
  Let $\nu$ be the unique equilibrium state of some H\"{o}lder-continuous function on $Y$.
  Then there is a unique measure of maximal relative entropy on $X$ over $\nu$.
\end{corollary}

We remark that the problem of obtaining a concrete description of the unique MMRE is still open.
It would be nice to have a description concrete enough to prove that $\mu$ is Bernoulli for instance.
But the question of whether $\mu$ is always Bernoulli is also open.

If $\nu$ is Markov, then $\rho$ (in the proof) is also Markov. So $\widetilde{\nu}$ is a hidden Markov measure. Since $\widetilde{Y}$ is usually strictly sofic, we cannot say that $\widetilde{\nu}$ is Markov. Nonetheless, the MMRE $\mu$ over the Markov $\nu$ is the MMRE over the hidden Markov $\widetilde{\nu}$ with respect to a class degree one map.
In order to describe $\mu$, we only need to describe MMREs over hidden Markov measures under class degree one maps.
For certain class degree one maps (maps with singleton clumps \cite{PQS-MaxRelEnt} for instance), this is doable.
For arbitrary class degree one factor codes, this seems to require further investigation, but it is hoped that this will turn out to be easier than using the original arbitrary infinite-to-one factor code directly.

\bibliographystyle{amsplain}

\providecommand{\bysame}{\leavevmode\hbox to3em{\hrulefill}\thinspace}
\providecommand{\MR}{\relax\ifhmode\unskip\space\fi MR }
\providecommand{\MRhref}[2]{%
  \href{http://www.ams.org/mathscinet-getitem?mr=#1}{#2}
}
\providecommand{\href}[2]{#2}

\end{document}